\documentclass{amsart}
\usepackage[utf8]{inputenc}

\usepackage[utf8]{inputenc}
\usepackage{lmodern}
\usepackage{dsfont}
\usepackage{enumitem}
\usepackage{autonum}
	\setenumerate{label={\normalfont(\roman*)}, itemsep=0em} 
       \usepackage{pgf,tikz}
\usepackage{amsfonts}
\usepackage{amsthm}
\usepackage{amsmath}
\usepackage{amssymb}
\usepackage{amscd}
\usepackage{mathrsfs}

\newtheorem{theorem}{Theorem}
\newtheorem{remark}[theorem]{Remark}
\newtheorem{proposition}[theorem]{Proposition}
\newtheorem{lemma}[theorem]{Lemma}

\title{Uniform large-scale $\e$-regularity for entropic optimal transport}
\author{Rishabh S. Gvalani}
\address[1]{D-MATH, ETH Z\"urich, R\"amistraße 101, 8001 Z\"urich, Switzerland}
\email{rgvalani@ethz.ch}
\author{Lukas Koch}
\address[2]{Mathematics Department, University of Sussex, Falmer Campus, BN1 9QH Brighton, United Kingdom}
\email{lukas.koch@sussex.ac.uk}

\newcommand{\e}{\varepsilon}
\newcommand\restr[2]{{
  \left.\kern-\nulldelimiterspace 
  #1 
  \vphantom{|} 
  \right|_{#2} 
  }}
\newcommand{\dd}{\mathrm{d}}
\newcommand{\supp}{\mathrm{supp}}
\newcommand{\R}{\mathbb R}

\begin{document}
\begin{abstract}
We study the regularity properties of the minimisers of entropic optimal transport providing a natural analogue of the $\e$-regularity theory of quadratic optimal transport in the entropic setting. More precisely, we show that if the minimiser of the entropic problem satisfies a gradient BMO-type estimate at some scale, the same estimate holds all the way down to the natural length-scale associated to the entropic regularisation.

Our result follows from a more general $\e$-regularity theory for optimal transport costs which can be viewed as perturbations of quadratic optimal transport. We consider such a perturbed cost and require that, under a certain class of admissible affine rescalings, the minimiser remains a local quasi-minimiser of the quadratic problem (in an appropriate sense) and that the cost of ``long trajectories'' of minimisers (and their rescalings) is small. Under these assumptions, we show that the minimiser satisfies an appropriate $C^{2,\alpha}$ Morrey--Campanato-type estimate which is valid up to the scale of quasi-minimality. 
\end{abstract}

\maketitle

In this paper we investigate regularity properties of minimisers of the entropic optimal transport problem:
\begin{align}
OT_\e(\lambda,\mu)=\min_{\pi\in \Pi(\lambda,\mu)} \int \lvert x-y\rvert^2 \dd \pi+ \e^2 \int \log\left(\frac{\dd \pi}{\dd(\lambda\otimes\mu)}\right) \dd \pi,
\end{align}
where $\lambda,\mu\in \mathscr M(\R^d)$ with $\lambda(\R^d)=\mu(\R^d)$ and $\Pi(\lambda,\mu)$ denotes the set of measures in $\mathscr M(\R^d\times \R^d)$ with marginals $\lambda$ and $\mu$\footnote{Note that entropic optimal transport is commonly formulated with $\e$ replacing $\e^2$ in the expression for $OT_\e$. We choose to use $\e^2$ as then $\e$ represents a length-scale in the problem.}.  Entropic optimal transport has received a lot of attention in recent years. This is due to the fact that it is possible to efficiently compute solutions of the minimisation problem using Sinkhorn's algorithm \cite{BenamouCarlier,Cuturi} and for $\e \ll 1$ the cost $OT_\e$ is close to $OT$, the cost of the quadratic optimal transport problem. In addition to this, the entropic problem has a rich structure with interesting connections to the Schr\"odinger bridge problem from physics and is thus of independent interest itself.  We refer the reader to the lecture notes \cite{Nutz} and review article~\cite{Leo14} for an introduction to the entropic optimal transport problem. 

For our purposes, a key insight is that (under regularity assumptions on the marginals) entropic optimal transport can be ``Taylor-expanded'' around quadratic optimal transport in the following manner
\begin{equation}
OT_\e(\lambda,\mu) = OT(\lambda,\mu)+ \frac d 2\e^2 \log(\e^{-2})+O(\e^2) \, .
\label{eq:eot}
\end{equation}
Higher-order terms in the expansion can be made explicit. The second-order expansion was first obtained in \cite{Erbar} and obtained under mild regularity assumptions in \cite{Carlier,Eckstein}. A third-order expansion was found in \cite{Conforti,Chizat}. Moreover, the contribution to the energy by the entropic part of $OT_\e$ and the contribution by the quadratic part was disentangled in \cite{Malamut}. We stress that all of these results are global while the property we will use for our result (see \eqref{eq:quasimin} below) is local. We do not aim for and do not obtain the precise form of the second-order term, but in order to obtain quasiminimality of entropic optimal transport at order $\e^2$ (rather than $\e^2 \log(\e^{-2})$) we do need to separate the quadratic and entropic contributions. In \cite{Malamut} this is obtained by utilising the dual formulation and Minty's trick. Here, we use convexity of the entropic part and a competitor based on the exact entropic minimiser on the torus. In order to formulate our main result for the entropic optimal transport problem , it is useful to define for $R>0$, $\#_R:= (B_R\times \R^d)\cup(\R^d\times B_R)$. Our main result is then the following regularity estimate.
\begin{theorem}\label{thm:entropicmain}
Suppose $\lambda$ and $\mu$ have $C^{0,\alpha}$-densities for some $\alpha\in (0,1)$ and $\pi$ is a minimiser of the entropic optimal transport problem~\eqref{eq:eot}. Define
\begin{align}
E(\pi,R):= \frac 1 {R^{d+2}}\int_{\#_R}\lvert x-y\rvert^2 \dd \pi, \, D_{\lambda,\mu}(R):=R^{2\alpha}([\lambda]_{\alpha,R}^2+[\mu]_{\alpha,R}^2)+|\lambda(0)-\mu(0)|^2. \label{eq:EDdef}
\end{align}
 Then, there exists some $\e_1>0$ such that if for some $R_0>0$ the densities of $\lambda$ and $\mu$ are bounded away from zero on $B_{R_0}$ and
\begin{align}
E(\pi,R_0)+D_{\lambda,\mu}(R_0)+\frac{\e^2}{R_0^2} <\e_1,
\label{eq:smallnessentropic}
\end{align}
then for any $r\leq R_0$ such that  $\left(\frac {r}{R_0}\right)\gg \left(\frac{\e}{R_0}\right)$, we have
\begin{align}
\min_{A\in \R^{d \times d},b\in \R^d} \frac 1 {r^{d+2}}\int_{\#_r} \lvert y-Ax-b\rvert^2 \dd \pi\lesssim E(\pi,R_0)+D_{\lambda,\mu}(R_0)+\frac{\e^2}{r^{2}} \, .
\end{align}
\end{theorem}
Note that there are two non-dimensional quantities in Theorem~\ref{thm:entropicmain}, the ratio $r/R_0$ of mesoscopic to macroscopic length-scales and the ratio $\e/R_0$ of entropic (microscopic) to macroscopic length-scales. As mentioned in the abstract, the above result follows from a more general regularity theory for costs that are perturbations of quadratic optimal transport. Indeed, we consider the transport problem
\begin{align}
\min_{\pi\in \Pi(\lambda,\mu)} \mathsf{c}(\pi).
\end{align}
where $c\colon \Pi(\lambda,\mu)\to \R$ is a given cost function and $\lambda,\mu, \pi, \Pi$ are as defined earlier. The most studied setting is when 
\[\mathsf{c}(\pi)=\int c(x-y) \, \dd \pi(x,y) \, ,
\label{eq:cost}
\]
 in which case, under mild regularity assumptions, minimisers exist and are of Monge-form, that is $\pi = (x,T(x))_{\#}\lambda$ for some map $T\colon \R^d\to \R^d$ (see \cite[Theorem 2.12]{Villani}). We think of $\mathsf{c}$ as a perturbation of the quadratic cost ($c(\cdot)=\lvert \cdot \rvert^2$ in~\eqref{eq:cost}) and will view minimisers of $\mathsf{c}$ as ``almost quasi-minimisers'' of the quadratic cost. In order to discuss our results further, we make our assumptions precise.

We start by fixing $\lambda,\mu \in \mathcal{M}(\R^d)$ with $\lambda(\R^d)=\mu(\R^d)$ such that $\lambda,\mu$ have densities which are bounded above and away from $0$ on some closed ball. We then consider a cost $\mathsf{c}:\Pi(\lambda,\mu) \to \R$. Consider now $\kappa\in K$, with $K$ a compact subset of $(0,\infty)$ containing $\lambda(0)^{-1}$, $b\in \R^d$, $A\in \R^{d\times d}$ positive-definite and symmetric,  and $\gamma\in G$ where $G$ is compact subset of $(0,\infty)$ which contains $\left(\frac{\lambda(0)}{\mu(0)}\right)^\frac 1 d$ in its interior\footnote{We conflate densities and measures here and from now on whenever it is convenient without further comment.}. Let us denote the set of all such $\mathsf{s}:=(A,b,\gamma,\kappa)$ by $\mathscr{S}$, the set of admissible rescalings. For any given $\mathsf{s}=(A,b,\gamma,\kappa)$, we also define the following objects:
\begin{gather}
Q(x,y)= (Q_1(x),Q_2(y))=(A^{-1}x, \gamma A(y-b))\\
\lambda_{\mathsf{s}}:=\kappa (Q_{1})_{\#}\lambda,\quad \mu_{\mathsf{s}}:= \kappa (Q_{2})_{\#}\mu,\label{eq:rescalings} \\
\pi_{\mathsf{s}}:=\kappa^2Q_{\#}\pi\, .
\end{gather}

 We now make certain assumptions on the cost $\mathsf{c}$ for all $R \leq R_0$ such that $(\e/R_0)^2 \ll (R/R_0)^2$, for some macroscopic length scale $R_0 \in (0,\infty)$. We call such $R$ admissible. We assume that there exist $C,\delta>0$, independent of the choice of $\mathsf{s}\in \mathscr{S}$ and of the choice of an admissible $R$, such that the following assumptions hold:
\begin{enumerate}
\item If $\pi \in \Pi(\lambda,\mu)$ is a minimiser of $\mathsf{c}$, for any $\mathsf{s}\in \mathscr{S}$,  $\pi_{\mathsf{s}} \in \Pi(\lambda_{\mathsf{s}},\mu_{\mathsf{s}}) $ is an almost quadratic quasi-minimiser. To be more precise, for any $\hat \pi =\tilde \pi+\restr{\pi_{\mathsf{s}}}{\#_R^c}\in \Pi(\lambda_{\mathsf{s}},\mu_{\mathsf{s}})$,
\begin{align}\label{eq:quasimin}
\int_{\#_R} \lvert x-y\rvert^2 \dd \pi_{\mathsf{s}}-\int \lvert x-y\rvert^2 \, \dd \tilde \pi \leq C\e^2 \pi_{\mathsf{s}}(\#_R)+\delta \int_{\#_{2R}} \lvert x-y\rvert^2 \dd\pi_{\mathsf{s}} \, .
\end{align}

\item The energy carried by long trajectories is small: If $\pi\in\Pi(\lambda,\mu)$ is a minimiser of $\mathsf{c}$, then there exist $\Lambda>0$ such that for all $\mathsf{s}\in \mathscr{S}$, if $E(\pi_{\mathsf{s}},2R)+D_{\lambda_{\mathsf{s}},\mu_{\mathsf{s}}}(2R)\ll1$ (see~\eqref{eq:EDdef}), then,
\begin{align}\label{eq:longTraj}
\frac 1 {R^{d+2}}\int_{\#_R\cap \{\lvert x-y\rvert \geq \Lambda R\}} \lvert x- y\rvert^2 \dd \pi_{\mathsf{s}}\leq C \delta E(\pi_{\mathsf{s}},2R).
\end{align}
\end{enumerate}

\begin{remark}
Almost quadratic quasi-minimality is usually formulated in the following form: $\pi\in \Pi(\Lambda,\mu)$ is an almost quadratic quasi-minimiser if there exist $C,\delta>0$ such that for any $\hat \pi= \tilde \pi + \restr{\pi}{\#_R^c}\in \Pi(\lambda,\mu)$,
\begin{align}\label{eq:quasiminStandard}
\int_{\#_R}|x-y|^2\dd\pi-(1+\delta)\int |x-y|^2 \dd\tilde \pi\leq C \e^2 \pi(\#_R). 
\end{align}
Choosing $\tilde \pi$ to be a quadratic optimiser for the marginal constraints imposed by $\hat \pi\in \Pi(\lambda,\mu)$ it is straightforward to see that \eqref{eq:quasiminStandard} implies \eqref{eq:quasimin}.
\end{remark}

We will we show in Section \ref{sec:application_to_entropic_optimal_transport} that all of the above assumptions are satisfied for the entropic optimal transport problem~\eqref{eq:eot}.  Under Assumptions (i) and (ii), our goal is to prove the following large-scale $\e$-regularity theorem.
\begin{theorem}\label{thm:main}
Suppose $\lambda$ and $\mu$ admit $C^{0,\alpha}$-densities for some $\alpha\in(0,1)$. Assume $\mathsf{c}$ satisfies Assumptions (i) and (ii), $\pi$ is a minimiser of $\mathsf{c}$ in $\Pi(\lambda,\mu)$, and $E,D_{\lambda,\mu}$ are as defined in Theorem~\ref{thm:entropicmain}. Then, for any $\alpha\in (0,1)$, there exists some $\e_1>0$ such that if for some $R_0>0$, $\lambda,\mu$ are bounded away from $0$ on $B_{R_0}$, and
\begin{align}
E(\pi,R_0)+D_{\lambda,\mu}(R_0)+\frac{\e^2}{R_0^2}+\delta <\e_1,
\label{eq:smallness}
\end{align}
then for any $\beta\in[0,\alpha]$, $r\leq R_0$ with $\left(\frac {r}{R_0}\right)^{2+2\beta}\gg \left(\frac{\e}{R_0}\right)^2$, we have
\begin{align}
\min_{A\in \R^{d \times d},b\in \R^d} \frac 1 {r^{d+2+2\beta}}\int_{\#_r} \lvert y-Ax-b\rvert^2 \dd \pi\lesssim R_0^{-2\beta}\left(E(\pi,R_0)+D_{\lambda,\mu}(R_0)\right)+\frac{\e^2}{r^{2+2\beta}} \, .
\end{align}.
\end{theorem}
We highlight that we explicitly allow for the case $\beta=0$ in Theorem \ref{thm:main}, in which case the regularity result holds all the way down to scale $O(\e)$. In particular, in the case of entropic regularisation (see Theorem~\ref{thm:entropicmain}), heuristically one would expect that at this scale the smoothing effect of the entropy dominates and smoothness propagates to arbitrarily small scales. However, we do not pursue this direction in this paper. We also note that in what follows we will write $D(R)$ for $D_{\lambda,\mu}(R)$ whenever it is clear from context which measures we are referring to.

Theorem \ref{thm:main} follows by carefully revisiting the regularity theory of quasi-minimisers studied in \cite{Otto2021}. The main difference in our approach is that in \eqref{eq:quasimin} quasi-minimality is viewed through the lense of $C^2$-scaling, rather than the $C^{2,\alpha}$-scaling studied in \cite{Otto2021}. The restriction to scales $\left(\frac r R\right)^{2+2\beta}\gg \left(\frac \e R\right)^2$ in Theorem \ref{thm:main} is a consequence of this. However, our assumption (and result) is natural in the context of entropic regularisation, where below the regularisation scale $\e$, the entropic term becomes dominant.

The main change in the proof of Theorem \ref{thm:main} relative to the regularity theory for $C^{2,\alpha}$-quasiminimisers in \cite{Otto2021} lies in the unavailability of the $L^p$-bounds for any $p>1$ for quasi-minimisers in our setting. This is replaced by Assumption (ii) which is satisfied by entropic minimisers. However, we still have to control trajectories of `medium' length. This is accomplished by the following lemma.

\begin{lemma}\label{lem:soft}
Fix $R>1$. Suppose $\pi$ is a quadratic quasi-minimiser in the sense that for any $\tilde \pi= \hat \pi + \restr{\pi}{\#_R^c}$ and some $\Delta_R>0$,
\begin{align}
\int_{\#_R} \lvert x-y\rvert^2 \dd \pi-\int \lvert x-y\rvert^2 \dd\hat\pi \leq \Delta_R \,
\end{align}
 and assume that
\begin{align}\label{cw03}
E(\pi,R)\ll \rho^{d+2}\ll R^{d+2} \, .
\end{align}
Then for any such $\rho$ with $D(\rho)\ll 1$,
\begin{align}\label{cw01}
\pi\big(\big\{(x,y)\in \#_{R-1}\cap \supp\,\pi\,\colon
\,\lvert x-y\rvert\ge \rho\big\}\big)\lesssim\frac{\Delta_R R^d}{\rho^{d+2}} \, .
\end{align}
\end{lemma}

Note that Lemma \ref{lem:soft} looks like a weak $L^{(d+2)^-}$-estimate, so one might expect to gain control of all trajectories of length $\gg E(\pi,R)$ from it. However, the restriction $\rho^{d+2}\ll R^{d+2}$, only enables to extract a control of trajectories of length at most $\lesssim R$.

Having Lemma \ref{lem:soft} and \eqref{eq:longTraj} at hand in order to replace $L^\infty$-bounds, as well as \eqref{eq:quasimin} to replace minimality it is straightforward to adapt the proof of the harmonic approximation result in \cite{Koch2023}.
\begin{proposition}\label{prop:harmonicApproximation}
Let $\pi$ be a minimiser for the cost $\mathsf{c}$. Assume $\lambda,\mu$ admit $C^{0,\alpha}$-densities in $B_{10}$, are bounded away from $0$ and $\lambda(0)=\mu(0)=1$. For every $\tau>0$, there exist $\epsilon(d,\tau)>0$ and $C_\tau=C(d,\tau),C=C(d)>0$
such that the following holds: If for $s\in \mathscr S$, $E(\pi_{\mathsf{s}},10)+D_{\lambda_{\mathsf{s}},\mu_{\mathsf{s}}}(10)+\e^2+\delta\leq \epsilon$ for some $\alpha\in(0,1)$, there exists $\phi$ harmonic such that
\begin{gather}
\int_{\#_1} |y-x-\nabla\phi(x)|^2 \dd \pi_{\mathsf{s}}\le\tau E(\pi_{\mathsf{s}},10)+C_\tau\left(D_{\lambda_\mathsf{s},\mu_\mathsf{s}}(10)+\e^2\right),\nonumber\\
\int_{B_1}|\nabla\phi|^2 \, \dd x\le C(E(\pi_{\mathsf{s}},10)+D_{\lambda_\mathsf{s},\mu_{\mathsf{s}}}(10)+ \e^2)).\label{eq:oneStep}
\end{gather}
\end{proposition}
We stress that in Proposition \ref{prop:harmonicApproximation}, $\epsilon$, $C_\tau$ and $C$ may be chosen independently of $\mathsf{s}\in \mathscr S$.

Proposition \ref{prop:harmonicApproximation} is the key to carrying out a Campanato iteration which ultimately gives us Theorem \ref{thm:main}. This part of the proof is similar to the one in \cite{Otto2021}.

\section{Proof of Theorem \ref{thm:main}}
\subsection{Controlling the mass of long trajectories}
In light of \eqref{eq:longTraj} we only control the energy contributed of very long trajectories, while implementing the strategy of \cite{Koch2023} will require to control the energy contribution by all trajectories of length at least $o(1)$. Hence, in light of \eqref{eq:longTraj}, we need to control the mass of trajectories of length $O(1)$. 

\begin{proof}[Proof of Lemma \ref{lem:soft}]
Since the statement is symmetric under exchanging the roles of $x$ and $y$, it is enough to show
\begin{align}
\pi\big(\big\{x\in B_{R-1}\;\mbox{and}\;\lvert x-y\rvert\ge \rho\big\}\big)\lesssim\frac{\Delta_R R^d}{\rho^{d+2}}.
\end{align}
Covering $B_{R-1}$ by $O(R^{-d} r^{-d})$ balls of radius $r$, which we think of as a small fraction
of $\rho$, and by translational symmetry, it is enough to show
\begin{align}
r^2\pi(B_r\times B_\rho^c)\lesssim \Delta_R\quad\mbox{provided}\;r\ll \rho.
\end{align}
Due to \eqref{cw03}, we may assume without loss of generality that $E(\pi,R)\ll r^{d+2}$. Covering the unit sphere by geodesic balls of radius $\alpha\ll 1$, and by rotational symmetry, it is enough to show that there exists
a universal (small and positive) $\beta$ with
\begin{align}\label{cw10}
&r^2\pi(B_r\times C_\rho)\lesssim \frac{\Delta_R}{R^d}\quad\mbox{provided}\;\beta\ll 1\\
&\mbox{where}\quad C_\rho:=B_\rho^c\cap\{{\textstyle\sqrt{\lvert y\rvert^2-y_1^2}}\le\beta y_1\},
\end{align}
where $\{\sqrt{\lvert y\rvert^2-y_1^2}\le\alpha y_1\}$ is a convex cone in direction $e_1$ of
opening angle $\alpha$.

\medskip

We now consider the ball $B'_r$ of radius $r$, $B_\bullet':=B_\bullet(3re_1)$.
By definition of $E(\pi,R)$, $E(\pi,R)\ll r^{d+2}$ implies
\begin{align}\label{cw11}
m:=\pi\big(B_r\times C_\rho)\leq\pi(B_r\times B_\rho^c)\ll \lvert B_r\rvert.
\end{align}
We now note that, 
\begin{align}
\pi(B_r'\times (B_{2r}')^c)\leq \pi(\{(x,y)\in B_r'\times \R^d\colon \lvert x-y\rvert \geq r\})\leq r^{-2} E(\pi,R)\ll \lvert B_r\rvert.
\end{align}
Consequently, using also that $D(r)$ is non-decreasing in $r$,
\begin{align}
&\lvert \pi(B_r'\times B_{2r}')-\lvert B_r\rvert\rvert \leq \lvert\pi(B_r\times \R^d)-\lvert B_r\rvert\rvert+\pi(B_r'\times (B_{2r})^c)\\
&\leq D(r) \lvert B_r\rvert+o(\lvert B_r\rvert)\ll \lvert B_r\rvert.
\end{align}

In particular, by continuity, there exists a radius $\tilde{r}\ll r$ such that
\begin{align}
\pi(B_{\tilde r}'\times  B_{2\tilde r}')=m.
\end{align}
This mass balance allows us to construct a competitor $\tilde\pi$ that instead
of sending the mass $m$ from $B_r$ into $C_\rho$
sends it into $B_{\tilde r}'$ and the excess mass from there into $C_\rho$. This involves
the initial measures
\begin{align}\label{cw07}
\int\zeta \,\dd \lambda_1:=\int_{B_r\times C_\rho}\zeta(x)\dd \pi,\quad
\int\zeta \,\dd \lambda_1':=\int_{B_{\tilde r}'\times B_{\tilde r}'}\zeta(x)\dd \pi
\end{align}
and the corresponding target measures
\begin{align}\label{cw08}
\int\zeta \,\dd \mu_1:=\int_{B_r\times C_\rho}\zeta(y)\dd \pi,\quad
\int\zeta \,\dd \mu_1':=\int_{B_{\tilde r}'\times B_{\tilde r}'}\zeta(y)\dd \pi,
\end{align}
which all have mass $m$. The competitor is defined as
\begin{align}
\int\zeta \, \dd \tilde\pi
&=\int_{((B_r\times C_\rho)\cup(B_{\tilde r}'\times B_{\tilde r}'))^c}\zeta \dd \pi\nonumber\\
&+\frac{1}{m}\int\int\zeta(x,y)\,\dd \lambda_1(x)\,\dd \mu_1'(y)
+\frac{1}{m}\int\int\zeta(x,y)\,\dd \lambda_1'(x)\,\dd \mu_1(y)
\end{align}
and clearly is non-negative; its admissibility can be inferred from
\begin{align}
\lefteqn{\int\zeta \dd(\pi-\tilde\pi)
=\int_{B_r\times C_\rho}\zeta \dd \pi
+\int_{B_{\tilde{r}}'\times B_{\tilde{r}}'}\zeta \dd \pi}\nonumber\\
&-\frac{1}{m}\int\int\zeta(x,y)\,\dd \lambda_1(x)\,\dd \mu_1'(y)
-\frac{1}{m}\int\int\zeta(x,y)\,\dd \lambda_1'(x)\,\dd \mu_1(y),
\end{align}
which also shows that the support of $\tilde\pi-\pi$ is contained
in $(B_r\cup B_{\tilde r}')\times(C_\rho\cup B_{2\tilde r}')$. Hence by quasi-minimality
\begin{align}\label{cw06}
\lefteqn{m\int_{B_r\times C_\rho}|x-y|^2\dd \pi\le m\Delta_R}\\
&+\int\int|x-y'|^2\,\dd \lambda_1(x)\,\dd \mu_1'(y')+\int\int|x'-y|^2\,\dd \lambda_1'(x')\,\dd \mu_1(y).
\end{align}

\medskip

Expanding the squares we have
\begin{align}
\frac{1}{2}\big(|x-y|^2+|x'-y'|^2-|x-y'|^2-|x'-y|^2\big)
=(x'-x)\cdot (y-y'),
\end{align}
which after elementary manipulations and using Young's inequality yields
\begin{align}
\lefteqn{(x'-x)\cdot(y-x)}\nonumber\\
&\le\frac{3}{2}|x'-x|^2+|x'-y'|^2+\frac{1}{2}\big(|x-y|^2-|x-y'|^2-|x'-y|^2)\\
\le& \frac{3}{2}|x'-x|^2+|x'-y'|^2+\frac{1}{2}\big(|x-y|^2-|x-y'|^2-|x'-y|^2)
\end{align}
We integrate this inequality with respect to
\begin{align}\label{cw08bis}
\frac{1}{m}I((x,y)\in B_r\times C_\rho)\dd \pi(x,y)\,\dd \lambda_1'(x')\,\dd \mu_1'(y')
\end{align}
and obtain by \eqref{cw06} and the definitions \eqref{cw07} and \eqref{cw08}
and the fact that all the measures have mass $m$
\begin{align}
&\int_{B_r\times C_\rho}\int(x'-x)\cdot(y-x)\,\dd \lambda_1'(x')\dd \pi(x,y)\le \frac{m}{2}\Delta_R\nonumber\\
&+\frac{3}{2}\int\int|x'-x|^2\,\dd \lambda_1'(x')\,\dd \lambda_1(x)
+\int\int|x'-y'|^2\,\dd \lambda_1'(x')\,\dd \mu_1'(y').
\end{align}
We note that, since $D(\rho)\ll 1$,
\begin{align}
\int_{B_{\tilde r}^\prime \times B_{\tilde r}^\prime}\lvert x-y\rvert^2 \dd \pi(x, y)\leq \tilde r^2 \pi(B_{\tilde r}^\prime \times \R^d)\ll r^2 m.
\end{align}

Since for $(x,y,x',y')$ in the support of \eqref{cw08bis}, that is,
for $(x,y,x',y')$ $\in B_r\times C_\rho\times B_{\tilde r}'\times B_{2\tilde r}'$,
we have by definition \eqref{cw10} of the cone $C_\rho$ provided $\alpha\ll 1$
\begin{align}
(x'-x)\cdot(y-x)\gtrsim r\rho,\quad
|x'-x|^2\lesssim r^2,\quad
|x'-y'|^2\lesssim r^2,
\end{align}

and since all measures have mass $m$, this yields by $r\ll \rho$,
\begin{align}
m r \rho\lesssim \Delta_R,
\end{align}
which in view of definition \eqref{cw11} amounts to \eqref{cw10}.
\end{proof}

\subsection{One-step improvement}
We want to use the harmonic approximation argument as explained in \cite{Koch2023} in order to obtain a one-step improvement. However as \cite{Koch2023} concerns minimisers of quadratic optimal transport, we need to modify the argument slightly. A careful inspection of \cite{Koch2023} shows that minimality is used at only two places: to control crossing trajectories \cite[Lemma 5]{Koch2023} and in order to localise minimality \cite[Lemma 2]{Koch2023}.

We begin by deriving the necessary replacements for \cite[Lemma 5, (66),(67)]{Koch2023}. We prove the estimate in Proposition \ref{prop:harmonicApproximation} for $\mathsf{s}=(\textup{Id},0,1,1)$ and drop the subscript $\mathsf{s}$ on $\pi_\mathsf{s}$. The proof for any other $\mathsf{s}\in \mathscr S$ is analogous and we remark that since the constants in Assumptions (i) and (ii) are independent of $\mathsf{s}$, all constants in the following are independent of the choice of $\mathsf{s}$. By scaling we may assume that $R_0=10$ and further that $E(\pi,10)+D(10)\ll 1$.  We first show that for any $\tau>0$ there is $C_\tau>0$ such that
\begin{align}\label{eq:Linfty1}
\int_2^3\int_{(x,y)\in \#_4\colon \exists t\; X(t)\in \partial B_R}\lvert x-y\rvert^2\dd \pi \dd R=\tau (E(\pi,10)+D(10))+C_\tau\e^2.
\end{align}
Here given $(x,y)\in \supp\; \pi$, we set $X(t)= (1-t)x+t y$.

Fix $\rho>0$ with $E(\pi,10)\ll \rho^{d+2}\ll 1$. We find
\begin{align}
&\int_2^3\int_{(x,y)\in \#_4\colon \exists t\; X(t)\in \partial B_R}\lvert x-y\rvert^2\dd \pi \dd R\\
\leq& \int_{(x,y)\in \#_4\colon \lvert x-y\rvert \geq 4\Lambda}\lvert x-y\rvert^2\dd \pi+\int_{(x,y)\in \#_4\colon \rho\leq \lvert x-y\rvert \leq 4\Lambda}\lvert x-y\rvert^2\dd \pi\\
&\quad+\int_2^3\int_{(x,y)\in \#_4\colon \lvert x-y\rvert \leq \rho \text{ and } \exists t\colon\; X(t)\in \partial B_R}\lvert x-y\rvert^2\dd \pi \dd R
\end{align}

Using the control of very long trajectories \eqref{eq:longTraj}, we control the first term by $C \delta E(\pi,10)$.
Using the quasi-minimality of $\pi$ \eqref{eq:quasimin} and Lemma \ref{lem:soft} with $R=5$, the second term is controlled by
\begin{align}
C \Lambda^2 \left(\frac{\e^2}{\rho^{d+2}}+\frac{\delta}{\rho^{d+2}}E(\pi,10)\right).
\end{align}
For future use, we note that we have shown for $E(\pi,10)\ll \rho^{d+2}\ll 1$,
\begin{align}\label{eq:Linfty}
\int_{\#_4\colon \rho\leq \lvert x-y\rvert}\lvert x-y\rvert^2\dd \pi\lesssim\frac{\e^2}{\rho^{d+2}}+\frac{\delta}{\rho^{d+2}} E(\pi,10).
\end{align}
Finally, we estimate the third term changing the order of integration by
\begin{align}
\rho \int_{\#_4} \lvert x-y\rvert^2\dd \pi.
\end{align}
Collecting estimates and choosing first $\rho$, then $\delta$ sufficiently small, this gives the desired estimate \eqref{eq:Linfty1}.
We further need to show
\begin{align}\label{eq:Linfty2}
\int_2^3 \pi(\{(x,y)\in \#_4\colon \exists t\; X(t)\in \partial B_R\})\dd R= o(1)+O(\e^2).
\end{align}
We find, again with $E(\pi,5)\ll \rho^{d+2}\leq \Lambda^{d+2}$,
\begin{align}
&\int_2^3 \pi(\{(x,y)\in \#_4\colon \exists t\; X(t)\in \partial B_R\})\dd R\\
&= \int_2^3 \pi(\{(x,y)\in \#_4\colon \exists t\; X(t)\in \partial B_R,\; \lvert x-y\rvert \leq \rho\})\dd R\\
&\quad+ \int_2^3 \pi(\{(x,y)\in \#_4\colon \exists t\; X(t)\in \partial B_R,\; \rho \leq \lvert x-y\rvert \leq 4\Lambda\})\dd R\\
&\qquad+ \int_2^3 \pi(\{(x,y)\in \#_4\colon \exists t\; X(t)\in \partial B_R,\; \lvert x-y\rvert \geq 4\Lambda\})\dd R
\end{align}
Changing the order of integration, the first term is estimated by
\begin{align}
\rho \pi(\#_3)\lesssim \rho.
\end{align}
The second term is controlled using Lemma \ref{lem:soft} with $R=5$ by
\begin{align}
C\Lambda^2\left(\frac{\e^2}{\rho^{d+2}}+\frac{\delta}{\rho^{d+2}} E(\pi,10)\right).
\end{align}
Finally, we control the third term using Markov's inequality and the control of very long trajectories \eqref{eq:longTraj} by
\begin{align}
\frac 1 {16\Lambda^2} \int_{\{(x,y)\in \#_4\colon \lvert x-y\rvert\geq 4\Lambda\}}\lvert x-y\rvert^2\dd \pi\lesssim \delta E(\pi,10).
\end{align}
This proves the claim \eqref{eq:Linfty2}.

We now turn to proving the replacement for \cite[Lemma 2]{Koch2023}: For any $\tilde \delta,\tau>0$, there is $C_{\tilde\delta},C_\tau>0$ such that
\begin{align}\label{eq:localMin}
\left(\int_{\#_R} \lvert x-y\rvert^2 \, \dd \pi(x,y)\right)^\frac 1 2\leq& W_2(\restr{\lambda}{B_R}+f_R,\restr{\mu}{B_R}+g_R)\\
&+C_{\tilde\delta}(\tau(E(\pi,10R)+D(10R))^\frac 1 2+ C_\tau\e).
\end{align}
Here $f_R$ and $g_R$ are defined via the relations
$$
\int_{\partial B_R} \xi \dd f = \int_{\{\exists t\colon X(t)\in \partial B_R\}} \xi(X(\sigma))\dd \pi, \quad \int_{\partial B_R} \xi \dd g = \int_{\{\exists t\colon X(t)\in \partial B_R\}} \xi(X(\tau))\dd \pi \quad 
$$
where $\sigma = \inf \{t>0\colon X(t)\in \bar B_R\}$ and $\tau = \sup \{t>0\colon X(t)\in \bar B_R\}$.

Let $(\tilde \lambda,\tilde \mu)$ be the marginals of $\restr{\pi}{\#_R}$. Let $\tilde \pi$ be the minimiser of $W_2(\tilde \lambda,\tilde \mu)$. Note that then $\tilde \pi+\restr{\pi}{\#_R^c}\in \Pi(\lambda,\mu)$. Hence by quasi-minimality \eqref{eq:quasimin},
\begin{align}
\int_{\#_R}\lvert x-y\rvert^2\dd \pi \leq&  W^2_2(\tilde \lambda,\tilde \mu)+C\e^2 \pi(\#_R)+C\delta E(\pi,10R).
\end{align}
Finally note that
\begin{align}
\pi(\#_R) \leq  \mu(B_4)+\lambda(B_4)\lesssim 1
\end{align}
to conclude for some $C>0$,
\begin{align}
\int_{\#_R}\lvert x-y\rvert^2\dd \pi \leq& W_2^2(\tilde \lambda,\tilde \mu)+C\e^2+C\delta R^{d+2}\int_{\#_{10R}}|x-y|^2\dd\pi .
\end{align}
Now write $\tilde \lambda = \restr{\lambda}{B_R}+\bar \lambda$, $\tilde \mu = \restr{\mu}{B_R}+\bar \mu$ and estimate using triangle inequality
\begin{align}
W_2(\tilde \lambda,\tilde\mu)\leq& W_2(\restr{\lambda}{B_R}+\bar \lambda,\restr{\lambda}{B_R}+f_R)+W_2(\restr{\lambda}{B_R}+f_R,\restr{\mu}{B_R}+g_R)\\
&+W_2(\restr{\mu}{B_R}+g_R,\restr{\mu}{B_R}+\bar \mu)\\
\leq& W_2(\bar \lambda,f_R)+W(\bar \mu,g_R)+W_2(\restr{\lambda}{B_R}+f_R,\restr{\mu}{B_R}+g_R).
\end{align}
In particular, by symmetry it suffices to estimate $W(\bar \lambda,f_R)$. Let $\pi_1$ be the plan that transports points according to the trajectory given by $\pi$, except that points entering $B_R$ are moved to the point where they cross the boundary. Formally,
\begin{align}
\int \xi(x,y)\dd \pi_1 = \int \xi(x,X(\sigma))I(x\in B_R, y\not\in B_R)\dd \pi.
\end{align}
Then
\begin{align}
W^2_2(\bar \lambda,f_R) \leq \int \lvert x-y\rvert^2 \dd \pi_1\leq \int_{(x,y)\colon \exists t\colon X(t)\in \partial B_R} \lvert x-y\rvert^2 \dd \pi.
\end{align}
Thus, applying \eqref{eq:Linfty1}, we obtain \eqref{eq:localMin} after collecting estimates.

With these items in hand, we can now follow \cite{Koch2023} replacing \cite[Lemma 2]{Koch2023} with \eqref{eq:localMin} and \cite[Lemma 5]{Koch2023} by \eqref{eq:Linfty1} and \eqref{eq:Linfty2} whenever necessary to prove Proposition \ref{prop:harmonicApproximation}. We remark that, since we are assuming $|\lambda(0)-\mu(0)|\leq \delta$, our definition of $D$ controls the notion of $D$ used in \cite{Koch2023}, see \cite[Lemma A.4]{Otto2021}.

With Proposition \ref{prop:harmonicApproximation} in hand, we now closely follow \cite[Proposition 1.16]{Otto2021} in order to obtain a one-step improvement. For the convenience of the reader, we provide a full proof.

\begin{theorem}\label{thm:onestep}
Let $\lambda$ and $\mu$ be measures of equal admitting $C^{0,\alpha}$-densities. Assume that $\lambda(0)=\mu(0)=1$.
Suppose $\pi\in \Pi(\lambda,\mu)$ is a minimiser of $\mathsf{c}$ and fix $\mathsf{s}\in \mathscr{S}$. Then for every $\beta\in(0,1)$, if $D_{\lambda_{\mathsf{s}},\mu_{\mathsf{s}}}(10R)+E(\pi_{\mathsf{s}},10R)+\frac{\e^2}{R^2}+\delta\ll 1 $, there exists $\theta\in(0,1)$, a symmetric matrix $A\in \R^{d\times d}$ with $\textup{det}\,A =1$ and a vector $b\in \R^d$ such that
\begin{align}\label{eq:Abestimates}
\lvert A-\textup{id}\rvert^2 + \frac 1 {R^2} \lvert b\rvert^2\lesssim E(\pi_{\mathsf{s}},10R)+D_{\lambda_{\mathsf{s}},\mu_{\mathsf{s}}}(10R)+ R^{-2}\e^2 
\end{align}
and
\begin{align}
E(\hat\pi,\theta R)\leq \theta^{2\beta} E(\pi_{\mathsf{s}},10R)+ C_\theta D_{\lambda_{\mathsf{s}},\mu_{\mathsf{s}}}(10R)+ C_\theta R^{-2}\e^2.
\end{align}
where $\hat\pi$ is obtained from $\pi$ as follows: Let $\gamma = \mu(b)^\frac 1 d$ , define $\hat{\mathsf{s}}=(A,b,\gamma,0)$, and define $\hat{\pi}=(\pi_{\mathsf{s}})_{\hat{\mathsf{s}}},\hat{\lambda}=(\lambda_{\mathsf{s}})_{\hat{\mathsf{s}}},\hat{\mu}=(\mu_{\mathsf{s}})_{\hat{\mathsf{s}}}$ (see~\eqref{eq:rescalings}). Equivalently, at the level of the densities we have
\begin{align}
\left(\begin{matrix} \hat x\\ \hat y\end{matrix}\right)= Q\left(\begin{matrix} x\\ y\end{matrix}\right) = \left(\begin{matrix} A^{-1} x\\ \gamma A(y-b)\end{matrix}\right)
\end{align}
and
\begin{align}
\hat \pi = Q_{\#}\pi_{\hat{\mathsf{s}}},\quad \hat\lambda(\hat x)=\lambda_{\mathsf{s}}(x),\quad \hat\mu(\hat y)=\gamma^{-d}\mu_{\mathsf{s}}(y) \, ,
\end{align}
where again $\hat{\lambda}(0)=1=\hat{\mu}(0)$. Moreover,
\begin{align}\label{eq:LEstimate}
\lvert \gamma-1\rvert^2 \lesssim E(\pi_{\mathsf{s}},10R)+D_{\lambda_{\mathsf{s}},\mu_{\mathsf{s}}}(10R)+\e^2 R^{-2} .
\end{align}
\end{theorem}
\begin{proof}
The argument can be carried out exactly as in \cite[Proposition 1.16]{Otto2021} with the following modification: the harmonic approximation is replaced by Proposition \ref{prop:harmonicApproximation}. We prove the result for $\mathsf{s}=(\mathrm{Id},0,1,1)$ and drop the subscript $\mathsf{s}$ on $\pi_\mathsf{s}$, $\lambda_{\mathsf{s}}$ and $\mu_{\mathsf{s}}$. The proof for general $\mathsf{s} \in \mathscr{S}$ follows in an identical manner using Assumptions (i) and (ii) and noting that since the constants there are independent of $\mathsf{s}$, all constants in the following are independent of $\mathsf{s}$.

By scaling, we may assume that $R=1$.

Let $\tau>0$ to be chosen later. Let $C_\tau$, $\e_\tau$ and $\phi$ be as in Proposition \ref{prop:harmonicApproximation}. We set
\begin{align}
b=\nabla \phi(0),\qquad A=e^{-\nabla^2\phi(0)/2}.
\end{align}
Then by elliptic regularity and Proposition \ref{prop:harmonicApproximation},
\begin{align}
\lvert b\rvert^2+\lvert \nabla^2\phi(0)\rvert^2 \leq& \sup_{B_{1/2}(0)} \lvert\nabla \phi\rvert^2+\lvert \nabla^2\phi\rvert^2 \lesssim \int_{B_1(0)} \lvert \nabla \phi\rvert^2\, \dd x\\
\lesssim& E(\pi,10)+D(10)+\e^2 .
\end{align}
Thus \eqref{eq:Abestimates} holds. In particular,
\begin{align}\label{eq:lEstimate}
\lvert \gamma -1 \rvert \leq \lvert b\rvert^{2\alpha} [\mu]_{\alpha,10}^2\lesssim (1+\e^2+E(\pi,10)^\alpha) [\mu]_{\alpha,10}^2.
\end{align}
Applying Young's inequality, this gives \eqref{eq:LEstimate}.

Assume that  $E(\pi,10)+D(10)+\e^2\ll \theta^2$. Then
\begin{align}
Q^{-1}(\#_\theta) = Q^{-1}(B_\theta\times \R^d)\cup Q^{-1}(\R^d\times B_\theta)= A B_\theta\times \R^d \cup \left(\R^d \times \gamma^{-1} A^{-1} B_\theta+b\right)
\end{align}
Due to \eqref{eq:Abestimates} and \eqref{eq:LEstimate}, it follows that $A B_\theta \subset B_{2\theta}$ and $\gamma^{-1}A^{-1} B_\theta\subset B_{2\theta}$.

We estimate
\begin{align}
\lvert \gamma-\textup{id}(y-b)\rvert \leq \lvert \gamma -1\rvert (\lvert y\rvert+\lvert b\rvert).
\end{align}
Further, we note, using the Taylor approximation
\begin{align}
\lvert A^{-2} -\left( \textup{id}- \nabla^2\phi(0)\right)\rvert \leq  C \sup_{B_{1/2}}\lvert \nabla^3\phi\rvert^2
\end{align}
we obtain
\begin{align}
\lvert b+A^{-2} x-(x+\nabla \phi(x))\rvert\leq& \lvert \nabla \phi(0)+\nabla^2\phi(0)x-\nabla\phi(x)\rvert+\sup_{B_{1/2}}\lvert \nabla^3\phi\rvert^2\\
\lesssim&  \sup_{B_{1/2}} \lvert\nabla^3\phi\rvert \lvert x\rvert + \sup_{B_{1/2}} \lvert \nabla^3\phi\rvert^2
\end{align}

We now compute
\begin{align}
&\theta^{d+2} E(\hat \pi,\theta)\\
=& \int_{\#_\theta \cap \{\lvert \hat x-\hat y\rvert \geq \Lambda\theta\}}\lvert \hat x-\hat y\rvert^2 \hat \dd \hat\pi+ \int_{\#_\theta\cap \{\lvert \hat x-\hat y\rvert \leq \Lambda\theta\}} \lvert \hat x-\hat y\rvert^2 \dd \hat\pi\\
\lesssim& \int_{\#_\theta\cap \{\lvert \hat x-\hat y\rvert \geq \Lambda\theta\}}\lvert \hat x-\hat y\rvert^2  \dd \hat\pi+|A|^2\int_{Q^{-1}(\#_\theta)\cap \{\lvert x-y\rvert\leq C(\Lambda)\theta\}}\lvert \gamma(y-b)-A^{-2} x\rvert^2\dd \pi\\
\lesssim& \int_{\#_\theta\cap \{\lvert \hat x-\hat y\rvert \geq \Lambda\theta\}}\lvert \hat x-\hat y\rvert^2 \dd \hat\pi+\\
&\quad+ \int_{\#_{2\theta}\cap \{x,y\in B_{C(\Lambda)\theta}\}} \lvert y-x-\nabla \phi(x)\rvert^2 \dd \pi + \lvert \gamma-1\rvert^2 \int_{\#_{2\theta}} \lvert y\rvert^2+\lvert b\rvert^2 \dd \pi\\
&\qquad+ \int_{\#_{3\theta}\cap \{x\in B_{C(\Lambda)\theta}\}}  \sup_{B_{1/2}} \lvert\nabla^3\phi\rvert^2 \lvert x\rvert^4 + \sup_{B_{1/2}} \lvert \nabla^3\phi\rvert^4\dd \pi.
\end{align}
The last two terms we estimate using \eqref{eq:Abestimates} and \eqref{eq:LEstimate}, as well as elliptic regularity and Proposition \ref{prop:harmonicApproximation},
\begin{align}
&\lvert \gamma-1\rvert^2 \int_{\#_{2\theta}} \lvert y\rvert^2+\lvert b\rvert^2 \dd \pi+ \int_{\#_{3\theta}\cap \{x\in B_{3\theta}\}}  \sup_{B_{1/2}} \lvert\nabla^3\phi\rvert^2 \lvert x\rvert^4 + \sup_{B_{1/2}} \lvert \nabla^3\phi\rvert^4 \lvert x\rvert^2\dd \pi\\
\lesssim& \left(\tau E(\pi,10)+C_\tau [\mu]_{\alpha,10}^2\right)(\theta^{d+2}+\theta^d \left(E(\pi,10)+[\lambda]_{\alpha,10}^2+[\mu]_{\alpha,10}^2+\e^2 )\right)\\
&\quad + \left(E(\pi,10)+([\lambda]_{\alpha,10}^2+[\mu]_{\alpha,10}^2)\right)\theta^{d+4}+\left(E(\pi,10)+([\lambda]_{\alpha,10}^2+[\mu]_{\alpha,10}^2)\right)^2.
\end{align}
Due to the invariance under affine transformations, we may apply \eqref{eq:longTraj} to the first term. Applying Proposition \ref{prop:harmonicApproximation} to the second term, and using Young's inequality, we deduce that
\begin{align}
E(\hat \pi,\theta)\lesssim \tau \theta^{-d+2} E(\pi,10)+\theta^2 E(\pi,10)+\theta^{-2}([\lambda]_{\alpha,10}^2+[\mu]_{\alpha,10}^2)+\theta^{-2}\e^2.
\end{align}
Choosing first $\theta$ small and then $\tau$ sufficiently small, we obtain the claimed result.
\end{proof}

\subsection{Campanato iteration}
We are finally ready to prove our main theorem.

\begin{proof}[Proof of Theorem \ref{thm:main}]

We focus on the case $\beta=0$ as the case $\beta>0$ is both easier and follows \cite{Otto2021} more closely.

 By making the following transformation 
\begin{align}
\lambda\to \lambda(0)^{-1}\lambda, \quad \mu\to \mu(0)^{-1}\mu\left(\left(\frac{\lambda(0)}{\mu(0)}\right)^{\frac1d}\cdot\right),
\end{align}
 we may assume that $\lambda(0)=\mu(0)=1$. Note in particular that the above rescaling 
 \[
 \bar{\mathsf{s}}=\left(\mathrm{Id},0,\left(\frac{\lambda(0)}{\mu(0)}\right)^{\frac1d},\lambda(0)^{-1}\right)
 \]
  lies in $ \mathscr{S}$ and thus Assumptions (i) and (ii) apply to it, i.e.  quasiminimality \eqref{eq:quasimin} is preserved and moreover, \eqref{eq:longTraj}. Before proceeding, we introduce the following notion of composition of scalings: given $\mathsf{s}_1=(A_1,b_1,\gamma_1,\kappa_1),\mathsf{s}_2=(A_2,b_2,\gamma_2,\kappa_2)$ such that $A_1,A_2$ are symmetric, positive-definite, $\mathrm{det}(A_1)=1=\mathrm{det}(A_2)$ and $\gamma_1,\gamma_2>0$, we define
  \begin{equation}
  \mathsf{s}_2 \diamond \mathsf{s}_1:= \left(A_2A_1,b_1+\gamma_2A_2 b_2 ,\gamma_2\gamma_1, \kappa_2\kappa_1\right) \, .
  \end{equation}
  The above notion of composition is chosen such that 
  \begin{align}
  (\lambda_{\mathsf{s}_1})_{\mathsf{s}_2}= \lambda_{\mathsf{s}_2 \diamond \mathsf{s}_1}, \quad (\mu_{\mathsf{s}_1})_{\mathsf{s}_2}=\mu_{\mathsf{s}_2 \diamond \mathsf{s}_1},\quad
  (\pi_{\mathsf{s}_1})_{\mathsf{s}_2}&= \pi_{\mathsf{s}_2 \diamond \mathsf{s}_1} \, ,
  \end{align}
  where $\lambda_{\mathsf{s}},\mu_{\mathsf{s}},\pi_{\mathsf{s}}$ are as defined in~\eqref{eq:rescalings}.

Set $R=R_0$.
 Before we can apply Theorem \ref{thm:onestep}, we need to check that 
\begin{equation}
 D_{\lambda_{\bar{\mathsf{s}}},\mu_{\bar{\mathsf{s}}}}(R)+E(\pi_{\bar{\mathsf{s}}},R)+\frac{\e^2}{R^2}+\delta\ll 1 \, .
 \label{eq:initialisation}
\end{equation}
By symmetry, we may assume that $\gamma = \left(\frac{\lambda(0)}{\mu(0)}\right)^\frac 1 d\geq 1$ (otherwise exchange the roles of $x$ and $y$),
\begin{align}
E(\pi_{\bar{\mathsf{s}}},R) = & \, \frac{1}{R^{d+2}}\int_{\#_{R}}|x-y|^2 \, \dd \pi_{\bar{\mathsf{s}}}\\
=& \frac{1}{\lambda(0)^2 R^{d+2}}\int_{\{B_{R}\times \R^d \}\cup \{\R^d \times B_{\gamma^{-1}R}\}}|x-\gamma y|^2 \, \dd \pi \\
\leq & \, \frac{2}{\lambda(0)^2 R^{d+2}}\int_{\{B_{R}\times \R^d \}\cup \{\R^d \times B_{\gamma^{-1}R}\}}|x-y|^2 \, \dd \pi \\&\,+ \frac{2(1-\gamma)^2}{\lambda(0)^2 R^{d+2}}\int_{\{B_{R}\times \R^d \}\cup \{\R^d \times B_{\gamma^{-1}R}\}}|y|^2 \, \dd \pi \, .\label{eq:pisisadmissible}
\end{align}
We estimate the terms on the right hand side separately. As $\gamma \geq 1$, the first term is controlled by $E(\pi,R)$. For the second term, we have, noting $D_{\lambda,\mu}(\gamma^{-1}R)\leq D_{\lambda,\mu}(R)$ as $\gamma\geq 1$,
\begin{align}
&\frac{2(1-\gamma)^2}{\lambda(0)^2 R^{d+2}}\int_{\{B_{R}\times \R^d \}\cup \{\R^d \times B_{\gamma^{-1}R}\}}|y|^2 \, \dd \pi\\
 \leq & \, \frac{2(1-\gamma)^2}{\lambda(0)^2 R^{d+2}}\left(\int_{B_{R}\times B_R }|y|^2 \, \dd \pi+\int_{B_R\times \R^d\setminus B_R}|y|^2\,\dd\pi + \int_{\R^d\times B_{\gamma^{-1}R}}|y|^2\,\dd\pi \right)\\
 \lesssim & \, \frac{(1-\gamma)^2}{\lambda(0)} + \frac{(1-\gamma)^2}{\lambda(0)^2}E(\pi,R)+\frac{(1-\gamma)^2}{\gamma^{2+2d} \lambda(0)}\, ,
\end{align}
However, note that $(1-\gamma)\lesssim D_{\lambda,\mu}(R)$. Hence, it remains to estimate the term $D_{\lambda_{\bar{\mathsf{s}}},\mu_{\bar{\mathsf{s}}}}(R)$ in~\eqref{eq:initialisation} which can clearly be controlled by $D_{\lambda,\mu}(R)$. Hence, Theorem \ref{thm:onestep} can be applied to $\pi_{\bar{s}}$ and we obtain
\begin{align}
E_1 =& E(\pi_{\mathsf{t}_1},\theta R)\leq \theta^{2\alpha} E(\pi_{{\bar{\mathsf{s}}}},R)+C_\theta R^{2\alpha}(([\lambda_{\bar{\mathsf{s}}}]_{\alpha,R}^2+[\mu_{\bar{\mathsf{s}}}]_{\alpha,R}^2))+C_\theta R^{-2}\e^2 \, ,
\end{align}
where $\mathsf{t}_1= \mathsf{s}_1\diamond \bar{\mathsf{s}}$ and $\mathsf{s}_1=(A_1,b_1,\gamma_1,\kappa_1)$ is the scaling obtained from Theorem~\ref{thm:onestep}. 

We will now show that
\begin{align}\label{eq:data}
&[\lambda_{\mathsf{t}_1}]_{\alpha,\theta R}+[\mu_{\mathsf{t}_1}]_{\alpha,\theta R}\\\leq & \, (1+C(E(\pi,R)^\frac 1 2 + R^\alpha([\mu_{\bar{\mathsf{s}}}]_{\alpha,R}+[\lambda_{\bar{\mathsf{s}}}]_{\alpha,R})+R^{-1}\e))([\lambda_{\bar{\mathsf{s}}}]_{\alpha,R}+[\mu_{\bar{\mathsf{s}}}]_{\alpha,R}) \, .
\end{align}
For $\mu_{\mathsf{t}_1}$, we have
\begin{align}
[\mu_{\mathsf{t}_1}]_{\alpha,\theta R}=& \gamma_1^{-d} \sup_{x,y\in B_{\theta R}} \frac{|\mu_{\bar{\mathsf{s}}}(\gamma_1^{-1}A_1^{-1}x+b_1)-\mu_{\bar{\mathsf{s}}}(\gamma_1^{-1}A_1^{-1} y+b_1)|}{|x-y|^\alpha}\\
\leq& \gamma_1^{-d} |\gamma_1 A_1^{-1}|^\alpha \sup_{x,y\in B_{\theta R}}\frac{|\mu(\gamma_1^{-1}A_1^{-1} x+b_1)-\mu(\gamma_1^{-1}A_1^{-1} y+b_1)|}{|(\gamma_1^{-1}A_1^{-1} x+b_1)-(\gamma_1^{-1}A_1^{-1} y+b_1)|^\alpha}\\
\leq& \gamma_1^{-(d+\alpha)}|A_1^{-1}|^\alpha \sup_{x',y'\in B_R}\frac{|\mu(x')-\mu(y')|}{|x'-y'|^\alpha}\\
=& \gamma_1^{-(d+\alpha)}|A_1^{-1}|^\alpha [\mu_{\bar{\mathsf{s}}}]_{\alpha,R}.
\end{align}
The argument for $\lambda_{\mathsf{t}_1}$ is similar.  Considering the estimates of $\gamma_1$ and $A_1$ we obtain from Theorem \ref{thm:onestep} this gives \eqref{eq:data}. Furthermore, we know that $\lambda_{\mathsf{t}_1}(0)=1=\mu_{\mathsf{t}_1}(0)$.

We would now like to reapply Theorem \ref{thm:onestep} for which we would need to justify that $\mathsf{t}_1$ is admissible. We shall do this later. For now, we set $r_k = \theta^k R$,  and assuming we can iterate Theorem \ref{thm:onestep}, we find a sequence of symmetric matrices $A_k$ with $\textup{det}\,A_k = 1$, a sequence of vectors $b_k$, real numbers $\gamma_k$ along with the associated scalings $\mathsf{s}_k=(A_k,b_k,\gamma_k,1)$, $\mathsf{t}_k=\mathsf{s}_k\diamond \mathsf{t}_{k-1}$. This allows us to define
\begin{align}
\lambda_k:=\lambda_{\mathsf{t}_k}, \, \mu_k:=\mu_{\mathsf{t}_k},\, \pi_k :=\pi_{\mathsf{t}_k}.
\end{align}
Noting that  $\lambda_k(0)=\mu_k(0)=1$ and, using Theorem~\ref{thm:onestep}, for $r_k \gg \e$, we have the estimate
\begin{align}
E_k :=E(\pi_k,r_k)\leq \theta^{2\alpha} E_{k-1}+C_\theta r_{k-1}^{2\alpha}([\lambda_{k-1}]_{\alpha,r_{k-1}}^2+[\mu_{k-1}]_{\alpha,r_{k-1}}^2)+C_\theta r_{k-1}^{-2}\e^2) \\
\lvert A_k-\textup{id}\rvert+\frac 1 {r_{k-1}^2}\lvert b_k\rvert^2 \lesssim E_{k-1}+r_{k-1}^{2\alpha}([\lambda_{k-1}]_{\alpha,r_{k-1}}^2+[\mu_{k-1}]_{\alpha,r_{k-1}}^2)+r_{k-1}^{-2}\e^2 \\
\lvert \gamma_k-1\rvert^2\lesssim E_{k-1}+r_{k-1}^{2\alpha}([\lambda_{k-1}]_{\alpha,r_{k-1}}^2+[\mu_{k-1}]_{\alpha,r_{k-1}}^2)+r_{k-1}^{-2}\e^2\label{eq:allEstimatesCampanato}
\end{align}

Note that as for \eqref{eq:data}, we have
\begin{align}\label{eq:data2}
&[\mu_k]_{\alpha,r_k}+[\lambda_k]_{\alpha,r_k}\\\leq&\, (1+C(E(\pi_{k-1},r_{k-1})^\frac 1 2 + r_{k-1}^\alpha([\mu_{k-1}]_{\alpha,r_{k-1}}+[\lambda_{k-1}]_{\alpha,r_{k-1}})+ r_{k-1}^{-1}\e))\\
&\times([\mu_{k-1}]_{\alpha,r_{k-1}}+[\lambda_{k-1}]_{\alpha,r_{k-1}})
\end{align}
We claim that
\begin{align}
[\mu_k]_{\alpha,r_k}+[\lambda_k]_{\alpha,r_k}\leq (1+\theta^{k\alpha}+C r_{k-1}^{-1}\e)([\mu_{k-1}]_{\alpha,r_{k-1}}+[\lambda_{k-1}]_{\alpha,r_{k-1}})\notag\\
E(\pi_k,r_k)\leq C(E(\pi,R)+R^{2\alpha}([\mu]_{\alpha,R}+[\lambda]_{\alpha,R}))+C r_k^{-2}\e^2.\label{eq:claim}
\end{align}
We prove \eqref{eq:claim} by induction. The case $k=1$ is clear, so suppose \eqref{eq:claim} holds for $k=1,\ldots,K-1$. By the induction hypothesis and \eqref{eq:data2},
\begin{align}
[\mu_{K-1}]_{\alpha,r_{K-1}}+[\lambda_{K-1}]_{\alpha,r_{k-1}}\leq \prod_{k=1}^{K-2} (1+\theta^{\alpha k}+C r_{k-1}^{-1}\e)([\mu]_{\alpha,R}+[\lambda]_{\alpha,R}).
\end{align}
Now note that for $r_K \gg \e$, i.e. $K \ll \frac{\log(\e/R)}{\log(\theta)}$,
\begin{align}
\prod_{k=1}^{K-2}(1+\theta^{\alpha k}+C (\theta^{k-1} R)^{-1} \e)\leq& \prod_{k=1}^{K-2}(1+\theta^{\alpha k}+C (\theta^{K-1} R)^{-1}\e)\\
\leq& \prod_{k=1}^{k_K}(1+2\theta^{\alpha k}) \prod_{k_K+1}^{K-2} (1+2C (\theta^{K-1} R)^{-1}\e).
\end{align}
Here $k_K = \max\{k\colon \theta^{\alpha k} \geq C(\theta^{K-1} R)^{-1}\e\}$. The first product is clearly finite with a bound independent of $k_K$. Regarding the second product, if $\theta^k R \geq C_0 \e$, we may bound it by
\begin{align}
(1+2C(\theta^{K-1} R)^{-1}\e)^K\leq (1+2C/C_0))^\frac{\log(C_0\e/R)}{\log \theta}
\end{align}
Elementary calculations now show that this product is bounded independent of $K$ as well, if $C_0$ is sufficiently large. Thus, we have shown that independently of $K$, as long as $\theta^K R \ll \e$,
\begin{align}
[\mu_{K-1}]_{\alpha,r_{K-1}}+[\lambda_{K-1}]_{\alpha,r_{K-1}}\leq C<\infty.
\end{align}
Then \eqref{eq:data2} and the induction hypothesis gives
\begin{align}
[\mu_K]_{\alpha,r_K}+[\lambda_K]_{\alpha,r_K}\leq (1+\theta^{K\alpha}+C r_{K-1}^{-1}\e)([\mu_{K-1}]_{\alpha,r_{K-1}}+[\lambda_{K-1}]_{\alpha,r_{K-1}}),
\end{align}
which is the first part of \eqref{eq:claim}. Note that a further consequence of our calculations is that for $r_k \gg \e$,
\begin{align}\label{eq:bounddata}
[\mu_{k}]_{\alpha,r_{k}}+[\lambda_{k}]_{\alpha,r_{k}}\lesssim [\mu]_{\alpha,R}+[\lambda]_{\alpha,R} \ll 1.
\end{align}
We turn to the second part of \eqref{eq:claim}. Note that using \eqref{eq:bounddata},
\begin{align}
\sup_{1\leq k \leq K}E(\pi_k,r_k)\leq& \theta^{2\alpha}(E(\pi,R)+\sup_{1\leq k\leq K-1} E(\pi_k,r_k)\\
&+C_\theta R^{2\alpha}([\lambda]_{\alpha,R}^2+[\mu]_{\alpha,R}^2)+C r_K^{-2}\e^2.
\end{align}
Since $\theta<1$, absorbing terms this gives the second part of \eqref{eq:claim}.

We need to now show that the iteration of Theorem \ref{thm:onestep} is justified. In order to this we need to prove that $\mathsf{t}_k =(B_k,d_k,\Gamma_k,\kappa_k) \in \mathscr{S}$ for all $k \leq K$, where
\begin{align}
B_k=&\, A_k A_{k-1}\dots A_1 , \quad  \Gamma_k = \left(\frac{\lambda(0)}{\mu(0)}\right)^{\frac1d}\Pi_{i=1}^k \gamma_i,\\
d_k =&\,\sum_{i=1}^{k}\left(\prod_{j=1}^{i-1}\gamma_{j} A_{j}\right)b_{i} ,\quad \kappa_k=\lambda(0)^{-1}\,.
\end{align}
Then,  combining \eqref{eq:allEstimatesCampanato}, \eqref{eq:claim} and \eqref{eq:bounddata}, as long as $r_k\gg \e$, we can insure that
\begin{align}
\lvert B_k-\textup{id}\rvert^2 \ll 1,\qquad \left\lvert \Gamma_k- \left(\frac{\lambda(0)}{\mu(0)}\right)^{\frac1d}\right \rvert \ll 1\,.
\end{align}
The second of the two above inequalities then ensures that $\mathsf{t}_k \in \mathscr{S}$ since $\Gamma_k \in G$ for all $k\leq K$ and thus our application of Theorem~\ref{thm:onestep} was justified. Furthermore, by a straightforward calculation, as in \cite{Goldman2020}, we can obtain
\begin{align}
\min_{A\in \mathbf{S}^d_+,b\in \R^d} \frac 1 {(\theta^k R)^{2+d}}\int_{\#_{\theta^k R}} \lvert y-Ax-b\rvert^2\dd \pi\lesssim& E(\pi,R)+([\lambda]_{\alpha,R}^2+[\mu]_{\alpha,R}^2)+r_k^{-2}\e^2\\
\leq& E(\pi,R)+D(R)+r_k^{-2}\e^2\\
\end{align}
Filling in the gaps between $r_k$ and $r_{k+1}$ in a routine fashion this completes the proof.
\end{proof}

 \section{Application to entropic optimal transport and the proof of Theorem~\ref{thm:entropicmain}} 
 \label{sec:application_to_entropic_optimal_transport}
 In this section, we will show that Assumptions (i) and (ii) which we have made for our general theory are valid for the entropic optimal transport problem ~\eqref{eq:eot} as a result of which the proof of Theorem~\ref{thm:entropicmain} follows. For a fixed $\lambda,\mu \in \mathcal{M}(\R^d)$ with $\lambda(\R^d)=\mu(\R^d)$, we note that the entropic cost $OT_\e(\lambda,\mu)$ (see~\eqref{eq:cost}) can be expressed as $\mathsf{c}_\e:\Pi(\lambda,\mu)\to \R$, where
 \begin{equation}
 OT_\e(\lambda,\mu) = \min_{\pi \in \Pi(\lambda,\mu)} \mathsf{c}_\e(\pi) \, ,
 \end{equation}
 where $\mathsf{c}_\e:\Pi(\lambda,\mu)\to \R$ is defined as follows
 \begin{equation}
\mathsf{c}_\e (\pi):= \int \lvert x-y\rvert^2 \dd \pi+ \e^2 \int \log\left(\frac{\dd {\pi}}{\dd(\lambda\otimes\mu)}\right) \dd \pi \, .
 \end{equation}
 To check the assumptions, we start by considering an admissible scaling $\mathsf{s}=(A,b,\gamma,\kappa)\in \mathscr{S}$ and $\pi_\e$ which minimises $\mathsf{c}_\e$. Then for $\lambda_{\mathsf{s}},\mu_{\mathsf{s}}$ and $\pi_{\e,\mathsf{s}} \in \Pi(\lambda_{\mathsf{s}},\mu_{\mathsf{s}})$ as defined in~\eqref{eq:rescalings}, we find
 \begin{align}
& \frac 1 {\kappa^2}\int |x-y|^2 \dd \pi_{\e,\mathsf{s}}+\frac{\e^2}{\kappa^2} \int \log\left(\frac{\dd  \pi_{\e,\mathsf{s}}}{\dd(\lambda_{\mathsf{s}} \otimes \mu_{\mathsf{s}})}\right) \dd  \pi_{\e,\mathsf{s}}\\
  =& \int |A^{-1}x-\gamma A(y-b)|^2 \dd \pi_\e+\e^2 \int \log \left(\frac{\dd \pi_\e}{\dd(\lambda\otimes\mu)}\right)\dd \pi_\e\\
 =& \int |A^{-1}x|^2-\gamma^2|x|^2+2\gamma\langle x,b\rangle \,\dd \lambda+\int |\gamma A(y-b)|^2-\gamma^2 |y|^2 \,\dd \mu+\int \gamma |x-y|^2 \dd \pi_\e\\
 &+\e^2 \int \log \left(\frac{\dd \pi_\e}{\dd(\lambda\otimes\mu)}\right)\dd \pi_\e.
 \end{align}
  Recognising that the first two integrals are null-Lagrangians, we have that $\pi_{\e,\mathsf{s}}$ is a minimiser of $\mathsf{c}_{\e \gamma^{-\frac12}}:\Pi(\lambda_{\mathsf{s}},\mu_{\mathsf{s}}) \to \R$. Given this information, we are in a position to prove Assumptions (i) and (ii).

 We first show that the energy of very long trajectories is small, i.e. establishing \eqref{eq:longTraj}, i.e. Assumption (ii). To this end, we prove the following proposition.

\begin{proposition}\label{prop:verylong}
Let $\pi_\e$ be the minimiser of entropic optimal transport at scale $\e>0$ and take $R>0$. Suppose $E(\pi_\e,5R)+D(5R)\ll 1$ and assume $\lambda,\mu$ are bounded away from $0$ in $B_R$. Then 
\begin{align}
\frac 1 {R^{d+2}}\int_{\#_{4R}\cap \{\lvert x-y\rvert \geq 7R\}} \lvert x-y\rvert^2\dd \pi_\e \lesssim e^{-\frac 1 {\e^2/R^2}}E(\pi_\e,5R) \, .
\end{align}
Similarly, for any admissible scaling $\mathsf{s}\in \mathscr{S}$ such that $E(\pi_{\e,\mathsf{s}},5R)+D_{\lambda_\mathsf{s},\mu_{\mathsf{s}}}(5R)\ll 1$, we have
\begin{equation}
\frac 1 {R^{d+2}}\int_{\#_{4R}\cap \{\lvert x-y\rvert \geq 7R\}} \lvert x-y\rvert^2\dd \pi_{\e,\mathsf{s}} \lesssim e^{-\frac 1 {\e^2/ R^2}}E(\pi_{\e,\mathsf{s}},5R) \, ,
\end{equation}
where the implicit constant is independent of the choice of scaling $\mathsf{s}\in \mathscr{S}$. Moreover, 
\begin{align}
\frac{1}{R^d}\pi_{\e,\mathsf{s}}(\#_{4R}\cap \{\lvert x-y\rvert \geq 7R\})\lesssim e^{-\frac 1 {\e^2/ R^2}}E(\pi_{\e,\mathsf{s}},5R) \, ,
\end{align}
where again the implicit constant is independent of the choice of scaling $\mathsf{s}\in \mathscr{S}$.
\end{proposition}

\begin{proof}
By scaling we may assume $R=1$. We prove the result only for the trivial scaling $\mathsf{s}=(\mathrm{Id},0,1,1)$ and note that for general $\mathsf{s}\in \mathscr{S}$ it follows from the fact that $\pi_{\e,\mathsf{s}}$ minimises $\mathsf{c}_{\e \gamma^{\frac12}}:\Pi(\lambda_{\mathsf{s}},\mu_{\mathsf{s}}) \to \R$ and the fact that $\kappa,\gamma $ lie in compact sets separated from $0$ and $\infty$. Let $\Lambda>0$ to be determined at a later stage.
We start by defining the following set:
\begin{align}
A(x,y) :=& \Big\{(x^\prime,y^\prime)\in \#_4\colon \lvert x-y\rvert^2+\lvert x^\prime-y^\prime\rvert^2-\lvert x^\prime-y\rvert^2-\lvert x-y^\prime\rvert^2 \geq 1,\\
&\quad \lvert x^\prime-y^\prime\rvert^2\leq \Lambda E(\pi,4)\Big \}.
\end{align}
Using the approximate cyclical montonicity of $\pi_\e$ (see \cite[Proposition 2.2]{Bernton2022}), we have
\begin{align}
&\frac{\dd \pi_\e(x,y)}{\dd(\lambda\otimes\mu)(x,y)}\frac{\dd \pi_\e(x^\prime,y^\prime)}{\dd(\lambda\otimes\mu)(x^\prime,y^\prime)}\\=&\,  e^{-\frac 1 {\e^2} \left(\lvert x-y\rvert^2+\lvert x^\prime-y^\prime\rvert^2-\lvert x^\prime-y\rvert^2-\lvert x-y^\prime\rvert^2\right)}
\times \frac{\dd \pi_\e(x,y^\prime)}{\dd(\lambda\otimes\mu)(x,y^\prime)}\frac{\dd \pi_\e(x^\prime,y)}{\dd(\lambda\otimes\mu)(x^\prime,y)}.
\label{eq:approxcm}
\end{align}
In particular, using the definition of $A(x,y)$ and \eqref{eq:approxcm}, we obtain
\begin{equation}\label{eq:pro4step1}
\begin{split}
&\int_{\#_4 \cap \{\lvert x-y\rvert\geq 7\}}\int_{A(x,y)}\lvert x-y\rvert^2\dd \pi_\e(x^\prime,y^\prime)\dd \pi_\e(x,y)\\
\leq& e^{-\frac 1 {\e^2}}\int \mathds 1_{\#_4 \cap \{\lvert x-y\rvert\geq 7\}\times A(x,y)}\lvert x-y\rvert^2\dd \pi_\e(x^\prime,y)\dd \pi_\e(x,y^\prime)\\
\lesssim& e^{-\frac 1 {\e^2}} \int \mathds 1_{\#_4 \cap \{\lvert x-y\rvert\geq 7\}\times A(x,y)} \lvert x-y^\prime\rvert^2+\lvert x^\prime-y^\prime\rvert^2+\lvert x^\prime-y\rvert^2 \dd \pi_\e(x,y^\prime)\dd \pi_\e(x^\prime,y)\\
\lesssim& e^{-\frac 1 {\e^2}}\pi_\e(\#_5)(2+\Lambda) E(\pi_\e,5)\lesssim e^{-\frac 1 {\e^2}}E(\pi_\e,5) \, .
\end{split}
\end{equation}

Given $(x,y)\in \#_4\cap \{\lvert x-y\rvert\geq 7\}$, assume without loss of generality $x\in B_4$ and consider the cone $C_\alpha(x,y)$ with vertex $x$ and aperture $\alpha\in(0,\pi)$ in direction $y-x$. Then for $(x^\prime,y^\prime)$ such that $x^\prime\in C_\alpha(x,y)\cap (B_2(x)\setminus B_{1}(x))\subset B_7$ with $\lvert x^\prime-y^\prime\rvert\leq \Lambda E(\pi_\e,5)$,
\begin{gather}
\lvert x^\prime-x\rvert\leq 2\\
\lvert x-y^\prime\rvert \leq \lvert x-x^\prime\rvert + \lvert x^\prime-y^\prime\rvert \leq 2+\lvert x^\prime-y^\prime\rvert\\
\lvert x^\prime-y\rvert \leq 2\textup{sin}(\alpha)+\lvert x-y\rvert-\textup{cos}(\alpha).
\end{gather}
We choose $\alpha$ sufficiently small to ensure
\begin{align}
\lvert x^\prime-y\rvert \leq \lvert x-y\rvert - \frac 3 4.
\end{align}
In particular, applying these bounds gives
\[
\begin{split}
 &\lvert x-y\rvert^2+\lvert x^\prime-y^\prime\rvert^2-\lvert x-y^\prime\rvert^2-\lvert x^\prime-y\rvert^2\\
\geq& |x-y|^2 + |x^\prime-y^\prime|^2 - (2+|x^\prime-y^\prime|)^2 - (|x-y|^2 -\frac34)^2\\
\geq& |x-y|^2+|x^\prime-y^\prime|^2 - 4 - 4|x^\prime - y^\prime| - |x^\prime-y^\prime|^2 - |x-y|^2 + \frac32 |x-y| - \frac9{16}  \\
=& \frac32|x-y| - 4  -4|x'-y'|- \frac9{16} \ge 1,
\end{split}
\]
as long as $\Lambda E(\pi_\e,5)\leq 1$.

Consequently,
\begin{align}
A(x,y)\supset \left\{(x^\prime,y^\prime)\colon x^\prime \in C_\alpha(x,y)\cap(B_2(x)\setminus B_{1}(x))\,\text{and}\,\lvert x^\prime-y^\prime\rvert^2 \leq \Lambda E(\pi_\e,5)\right\}.
\end{align}
For any cone $C_\alpha$ with aperture $\alpha$, centered at a point in $B_5$, as $D(5)\ll 1$,
\begin{align}
\pi_\e(C_\alpha\cap (B_2(x)\setminus B_{1}(x))\times \R^d)\gtrsim \alpha.
\end{align}
Moreover,
\begin{align}
& \pi_\e\left((\left(C_\alpha \cap (B_2(x)\setminus B_{1}(x))\times \R^d\right)\cap \{(x^\prime,y^\prime)\colon\lvert x^\prime-y^\prime\rvert^2\geq \Lambda E(\pi_\e,5)\}\right)\\
\leq& (\Lambda E(\pi_\e,5))^{-1}\int_{C_\alpha \cap (B_2(x)\setminus B_{1}(x))\times \R^d} \lvert x^\prime-y^\prime\rvert^2\dd \pi_\e \leq \Lambda^{-1}.
\end{align}
Thus, we deduce for some $c,c_1>0$,
\begin{align}
\pi_\e\left(\left(C_\alpha \cap (B_2(x)\setminus B_{1}(x))\times \R^d\right)\cap \{(x^\prime,y^\prime)\colon\lvert x^\prime-y^\prime\rvert \leq 1\}\right)\geq c\alpha-c_1\Lambda^{-1}\geq \frac{c\alpha} 2,
\end{align}
where to obtain the last inequality we chose $\Lambda>\frac{2c_1}{c \alpha}$. In particular, this shows
\begin{align}
&\int_{\#_4\cap \{\lvert x-y\rvert \geq 7\}}\int_{A(x,y)}\lvert x-y\rvert^2 \dd \pi_\e(x^\prime,y^\prime)\dd \pi_\e(x,y)\\
\geq& \int_{\#_4\cap \{\lvert x-y\rvert \geq 7\}}\lvert x-y\rvert^2\pi_\e\Big(\left(C_\alpha(x,y)\cap (B_2(x)\setminus B_{1}(x))\times \R^d\right) \\
&\quad\cap \{(x^\prime,y^\prime)\colon\lvert x^\prime-y^\prime\rvert \leq 1\}\Big) \dd \pi_\e(x,y)\\
\geq& \frac{c\alpha} 2 \int_{\#_4 \cap \{\lvert x-y\rvert \geq 7\}}\lvert x-y\rvert^2\dd \pi_\e.
\end{align}
Combining the latter with \eqref{eq:pro4step1} yields
\begin{align}
\int_{\#_4 \cap \{\lvert x-y\rvert \geq 7\}}\lvert x-y\rvert^2\dd \pi_\e\lesssim e^{-\frac 1 {\e^2}} E(\pi_\e,5).
\end{align}
For the moreover part, we proceed as follows
\begin{align}
& \, \pi_\e(\{\#_4\cap \{|x-y|\geq 7\})  \\
\lesssim & \, \int_{\#_4 \cap \{\lvert x-y\rvert \geq 7\}}\lvert x-y\rvert^2\, \dd \pi_\e \lesssim e^{-\frac {1}{\e^2}} \int_{\#_5} |x-y|^2 \, \mathrm{d}\pi_\e \, .
\end{align}
This completes the proof.
%
\end{proof}
\subsection{Local quasiminimality of entropic optimal transport}
In this subsection, we will show the quasiminimality required by Assumption (i). Before we present the main result of this section, we introduce the following set:
\begin{equation}
P_R := (B_R \times B_{\Lambda R}) \cup (B_{\Lambda R} \times B_{ R}) \, ,
\end{equation}
for some $\Lambda >1$, noting that $P_R \subset \#_R$. The exact result we will prove takes the following form:
\begin{proposition}\label{prop:quasiMin1}
Suppose $\pi_\varepsilon$ is the minimiser of entropic optimal transport at scale $\e>0$ and fix $R\gg \e$. Assume $\lambda,\mu$ admit $C^{0,\alpha}$-densities and are bounded away from $0$ on $B_R$. Further assume $E(\pi,R)+D_{\lambda,\mu}(R)\ll 1$. Let $(\bar \lambda,\bar \mu)$ be the normalised marginals of $\restr{\pi_\e}{P_R}$. Then, choosing $\Lambda=11/4$, for any $\delta\in(0,1)$ there exists a $C_1=C_1(\delta,R,\lambda(0),\mu(0))<\infty$, such that
\begin{align}\label{eq:quasiminEntro}
\int_{\#_R}\lvert x-y\rvert^2 \dd \pi_\e\leq & \, \pi_\e (P_R)OT(\bar\lambda,\bar\mu)+C_1 \pi_\e(\#_{\Lambda R}) \e^2 \\
& \, + \delta \int_{\#_{2R}}|x-y|^2\, \mathrm{d}\pi_{\e}\, .
\end{align}
Moreover, it is possible to ensure that $C_1$ is increasing as a function of $R$.
Similarly, for any admissible scaling $\mathsf{s}\in \mathscr{S}$, we have
\begin{align}\label{eq:quasiminEntros}
\int_{\#_R}\lvert x-y\rvert^2 \dd \pi_{\e,\mathsf{s}}\leq & \, \pi_{\e,\mathsf{s}} (P_R)OT(\bar\lambda_{\mathsf{s}},\bar\mu_{\mathsf{s}})+C_1 \pi_{\e,\mathsf{s}}(\#_{\Lambda R}) \e^2 \\
& \, + \delta \int_{\#_{2R}}|x-y|^2\, \mathrm{d}\pi_{\e,\mathsf{s}}\, ,
\end{align}
where $(\bar \lambda_{\mathsf{s}},\bar \mu_{\mathsf{s}})$ are the normalised marginals of $\restr{\pi_{\e,\mathsf{s}}}{P_R}$.
\end{proposition}
The above result can be translated into the form of Assumption (i) by making the following observation: given any $\hat{\pi}=\tilde{\pi} + \restr{\pi_\e}{\#_R^c}\in \Pi(\lambda,\mu)$, then we know that $\tilde{\pi}\in \pi_\e (\#_R) \Pi(\tilde{\lambda},\tilde{\mu})$ where $\tilde{\lambda},\tilde{\mu}$ are the normalised marginals of $\left.\pi_\e\right|_{\#_R}$.  We then estimate using the triangle inequality,
\begin{align}\label{eq:postProcess}
OT(\bar \lambda,\bar \mu)=& \frac{\bar \lambda(R^d)^2}{\tilde \lambda(\R^d)^2} OT\left(\frac{\tilde \lambda(\R^d)}{\bar \lambda(\R^d)}\bar \lambda ,\frac{\tilde \mu(\R^d)}{\bar \mu(\R^d)}\bar \mu \right)\\
\leq&\frac{\bar \lambda(\R^d)^2}{\tilde \lambda(\R^d)^2}\left( OT\left(\frac{\tilde\lambda(\R^d)}{\bar \lambda(\R^d)}\bar \lambda,\tilde\lambda\right)+OT(\tilde\lambda,\tilde\mu)+OT\left(\tilde \mu,\frac{\tilde \mu(\R^d)}{\bar \mu(\R^d)}\bar\mu\right)\right)
\end{align}
Note that using Proposition \ref{prop:verylong}, as $\e \ll R$,
\begin{align}
0\leq \tilde\lambda(\R^d)- \bar \lambda(\R^d)\leq\pi_\e(\{(x,y)\in \#_R\colon |x-y|\geq \Lambda R\})\leq \delta\frac{1}{R^2} \int_{\#_{2R}}|x-y|^2 \, \dd \pi_\e \, .
\end{align}
Consequently, we can ensure
\begin{align}
1\leq \frac{\tilde \lambda(\R^d)}{\bar \lambda(\R^d)}\leq 1+ \delta\frac{1}{R^2 \pi_\e (P_R)} \int_{\#_{2R}}|x-y|^2 \, \dd \pi_\e.
\end{align}
In particular, it remains to estimate the first and third term on the right hand side of \eqref{eq:postProcess}. We first define the following set
\begin{equation}
A:= \{(x,y) \in \R^d \times \R^d : x\in \R^d\setminus B_{\Lambda R} , y \in B_R \}\, .
\end{equation}
We then proceed to estimate the term as follows:
\begin{align}
OT\left(\frac{\tilde \lambda(\R^d)}{\bar \lambda(\R^d)}\bar \lambda,\tilde\lambda\right)\leq& OT\left(\left(\frac{\tilde \lambda(\R^d)}{\bar \lambda(\R^d)}-1\right)\bar \lambda,\tilde \lambda-\bar \lambda\right)\\
\leq& OT\left(\left(\frac{\tilde \lambda(\R^d)}{\bar \lambda(\R^d)}-1\right)\bar \lambda,\Pi_y \pi_\e\big|_A )\right)\\
&+OT(\Pi_y \pi_\e\big|_A ,\tilde\lambda-\bar \lambda) \, .\label{eq:beforeproofest}
\end{align}
We treat the two terms on the right hand side separately. For the second term, we note that $\pi_\e\big|_A$ is a competitor since $\Pi_x \pi_\e\big|_A=\tilde{\lambda}-\bar{\lambda}$. Thus, we can control it in the following manner
\begin{equation}
OT\left(\left(\frac{\tilde \lambda(\R^d)}{\bar \lambda(\R^d)}-1\right)\bar \lambda,\Pi_y \pi_\e\big|_A )\right) \leq \int_{A}|x-y|^2 \, \mathrm{d}\pi_\e \leq \delta \int_{\#_{2R}}|x-y|^2 \, \mathrm{d}\pi_\e \, ,
\end{equation}
where in the last step we have applied Proposition~\ref{prop:verylong} and used the fact that $\e \ll R$. For the first term on the right hand side of~\eqref{eq:beforeproofest}, we note that any coupling $\pi\in \Pi\left(\left(\frac{\tilde \lambda(\R^d)}{\bar\lambda(\R^d)}-1\right)\bar \lambda,\Pi_y \pi_\e\big|_A )\right)$ must be supported on $B_{\Lambda R}\times B_R$ from which it follows that
\begin{equation}
 OT\left(\left(\frac{\tilde \lambda(\R^d)}{\bar \lambda(\R^d)}-1\right)\bar \lambda,\Pi_y \pi_\e\big|_A )\right) \lesssim  R^2 (\tilde{\lambda}(\R^d)-\bar{\lambda}(\R^d))\lesssim \delta \int_{\#_{2R}}|x-y|^2\, \dd \pi_\e \, .
\end{equation} 
It then follows that we have the estimate
\begin{align}
&\, \pi_\e (P_R) OT(\bar\lambda,\bar\mu) \\
\lesssim & \, OT(\tilde{\lambda},\tilde{\mu}) + \delta \int_{\#_{2R}}|x-y|^2\, \dd \pi_\e + \frac{\delta}{R^2} \int_{\#_{2R}}|x-y|^2\, \dd \pi_\e OT\left(\frac{\tilde \lambda(\R^d)}{\bar \lambda(\R^d)}\bar \lambda ,\frac{\tilde \mu(\R^d)}{\bar \mu(\R^d)}\bar \mu \right) \\
\lesssim & \, OT(\tilde{\lambda},\tilde{\mu}) + \delta \int_{\#_{2R}}|x-y|^2\, \dd \pi_\e \, ,
\end{align}
where in the last step we have used the fact that $\bar{\lambda},\bar{\mu}$ are supported on $B_{\Lambda R}$ and so 
\begin{equation}
OT\left(\frac{\tilde \lambda(\R^d)}{\bar \lambda(\R^d)}\bar \lambda ,\frac{\tilde \mu(\R^d)}{\bar \mu(\R^d)}\bar \mu \right)\lesssim R^2(\tilde{\lambda}(\R^d) + \tilde{\mu}(\R^d)) \leq R^2\, .
\end{equation}

Moreover, for $R\leq R_0$, $\pi_\e(\#_{\Lambda R})\lesssim \pi_\e(\#_R)$. Thus, \eqref{eq:quasiminEntro} can be reduced to the form of Assumption (i). We now provide the proof of the result.

\begin{proof}
As before, we prove the result only for the trivial scaling $\mathsf{s}=(\mathrm{Id},0,1,1)$, since the proof in the general case is similar. Let $(\bar \lambda,\bar \mu)$ denote the normalised marginals of $\restr{\pi_\e}{P_R}$, i.e.
\begin{equation}
\bar{\lambda}(A)=\frac{\restr{\pi_\e}{P_R}(A\times \R^d)}{\pi_\e(P_R)}\, ,
\end{equation}
and analogously for the other marginal. Let $\bar \pi = \pi_\e(P_R)\textup{argmin}\, OT_\e(\bar \lambda,\bar \mu)$. Note that $\pi_\e (P_R) \bar \lambda\leq \lambda$, $\pi_\e (P_R)\bar \mu \leq \mu$. Consider $\hat \pi = \restr{\pi_\e}{P_R^c} + \bar \pi$. Note that
\begin{align}
\Pi_x(\hat \pi) = \Pi_x\restr{\pi_\e}{P_R^c} + \bar \lambda = \lambda,
\end{align}
and similarly $\Pi_y(\hat \pi) = \mu$. Thus, $\hat \pi$ is valid competitor for $\pi_\e$ and we find
\begin{align}
&\int \lvert x-y\rvert^2 \dd \pi_\e + \e^2 H(\pi_\e | \lambda\otimes \mu) \leq \int \lvert x-y\rvert^2 \dd \hat \pi + \e^2 H(\hat \pi|\lambda\otimes \mu).\label{eq:entropicMin}
\end{align}
Set $A = \supp(\restr{\pi_\e}{P_R^c})\cap \supp(\bar \pi)$. Using convexity of the function $x \mapsto x\log x$, we obtain the following estimate:
\begin{align}
&\int_{A} \log\left(\frac{\dd(\pi_\e+\bar \pi)}{\dd(\lambda\otimes\mu)}\right)\dd(\pi_\e+\bar \pi)\\
\leq& \frac 1 2 \int_{A} \log\left(\frac{2\dd \pi_\e}{\dd(\lambda\otimes\mu)}\right)2\dd \pi_\e + \frac 1 2 \int_{A} \log\left(\frac{2\, \dd \bar\pi}{\dd(\lambda\otimes\mu)}\right)2\, \dd \bar\pi\\
=& \int_{A} \log\left(\frac{\dd \pi_\e}{\dd(\lambda\otimes\mu)}\right)\dd \pi_\e + \int_{A} \log\left(\frac{\, \dd \bar\pi}{\dd(\lambda\otimes\mu)}\right)\, \dd \bar\pi+\int_{A} \log(2)\dd(\pi_\e+\bar \pi)\\
\leq& \int_{A} \log\left(\frac{\dd \pi_\e}{\dd(\lambda\otimes\mu)}\right)\dd \pi_\e + \int_{A} \log\left(\frac{\, \dd \bar\pi}{\dd(\lambda\otimes\mu)}\right)\, \dd \bar\pi+\log(2) \left(\pi_\e(A)+\bar \pi(A) \right) \,.
\end{align}
Combining the above estimate with \eqref{eq:entropicMin}, we arrive at:
\begin{align}\label{eq:intermed1}
&\int_{\#_R} \lvert x-y\rvert^2 \dd \pi_\e + \e^2\int_{P_R} \log\left(\frac{\dd \pi_\e}{\dd(\lambda\otimes\mu)}\right)\dd \pi_\e \nonumber\\
\leq& \int \lvert x-y\rvert^2 \dd \bar \pi + \e^2\int \log\left(\frac{\, \dd \bar\pi}{\dd(\bar \lambda\otimes\bar\mu)}\frac{\dd(\bar \lambda\otimes\bar\mu)}{\dd(\lambda\otimes\mu)}\right)\, \dd \bar\pi+ \log(2)\e^2\left(\pi(A)+\bar\pi(A)\right)\nonumber\\
\leq& \pi_\e(P_R) OT_\e(\bar \lambda,\bar \mu)+2\log(2)\e^2\pi_\e (\#_{\Lambda R})- 2\e^2 \pi_\e(P_R)\log\pi_\e(P_R) \\& \, + \int_{\#_R\setminus P_R} \lvert x-y\rvert^2 \dd \pi_\e  \, ,
\end{align}
where we have used the fact that $\bar{\pi}(\R^d \times \R^d) \leq \pi_\e (P_R)$ and that $\pi_\e(A)\leq \pi_\e(B_{\Lambda R}\times B_{\Lambda R}) \leq \pi_\e (\#_{\Lambda R})$. Now, comparing entropic optimal transport to quadratic transport, we obtain for some constant $C_{\mathrm{ T}}<\infty$
\begin{align}
OT_\e(\bar \lambda,\bar \mu)\leq OT(\bar \lambda,\bar\mu)+\frac d2\e^2 \log(\e^{-2})+C_{\mathrm{T}}\e^2,
\end{align}
which leaves us with
\begin{align}
&\int_{\#_R} \lvert x-y\rvert^2 \dd \pi_\e + \e^2\int_{P_R} \log\left(\frac{\dd \pi_\e}{\dd(\lambda\otimes\mu)}\right)\dd \pi_\e \\
\leq & \, \pi_\e(P_R) OT(\bar \lambda,\bar \mu)+\pi_\e(P_R) \frac d2\e^2 \log(\e^{-2})\\&\,+C_{\mathrm{T}}\pi_\e(P_R)\e^2 \\ &\, +2\log(2)\e^2\pi_\e (\#_{\Lambda R})- 2\e^2 \pi_\e(P_R)\log\pi_\e(P_R) \\& \, + \int_{\#_R\setminus P_R} \lvert x-y\rvert^2 \dd \pi_\e   \, .
\label{eq:fullcostbound}
\end{align}
If we choose $\Lambda =11/4$, the last term on the right hand side can be controlled using Proposition \ref{prop:verylong} by $C_2 \delta \int_{\#_{2R}}|x-y|^2\, \, \mathrm{d}\pi_\e$ by choosing  $\e/R \leq \delta$. We will now derive a careful lower bound on the entropic contribution of $\pi_\e$ on $P_R$. Using the convexity of $x\mapsto x\log x$, we have that for any $\pi \ll \lambda \otimes \mu$, the following bound holds true 
\begin{align}
\int_{P_R} \log\left(\frac{\dd \pi_\e}{\dd(\lambda\otimes\mu)}\right)\frac{\dd \pi_\e}{\dd(\lambda\otimes\mu)}\dd(\lambda\otimes\mu)\geq \int_{P_R} \left(1+\log\left(\frac{\dd \pi}{\dd(\lambda\otimes\mu)}\right)\right)\dd \pi_\e - \pi(P_R) \, .
\label{eq:entropiclowerbound1}
\end{align}

We now make the following choice:
\[\pi = \pi_\e(P_R)\frac{\mathbf{1}_{B_R\times B_R} e^{-\frac{\delta\lvert y-x\rvert^2}{(1+\delta)\e^2}}\dd(\lambda\otimes \mu)}{\int_{B_R\times B_R} e^{-\frac{\delta\lvert y-x\rvert^2 }{(1+\delta)\e^2}}\dd(\lambda \otimes \mu)}\,.\]
 We can then obtain
\begin{align}
&\e^2\int_{P_R} \log\left(\frac{\dd \pi}{\dd(\lambda\otimes\mu)}\right)\dd \pi_\e \\=&\, -\frac{\delta}{1+\delta} \int_{B_R\times B_R} \lvert y-x\rvert^2\dd \pi_\e + \e^2\log\left(\frac{1}{\int_{B_R \times B_R} e^{-\frac{\delta\lvert y-x\rvert^2}{(1+\delta)\e^2}}\dd(\lambda\otimes \mu)}\right)\pi_\e(P_R) 
\\&\, +\e^2 \pi_\e(P_R)\log(\pi_\e(P_R)) \, .
\label{eq:lower_bound_entropy}
\end{align}
For the second term on the right hand side, we estimate the integral in the argument of the $\log$ as follows
\begin{align}
&\int_{B_R \times B_R} e^{-\frac{\delta\lvert y-x)\rvert^2}{(1+\delta)\e^2}}\,\dd(\lambda\otimes \mu)\\
=&\, \frac{\lambda(B_R)\mu(B_R)}{|B_R|^2}\int_{B_R \times B_R} e^{-\frac{\delta\lvert y-x)\rvert^2}{(1+\delta)\e^2}}\,\dd(x\otimes y) \\&\,+ \int_{B_R \times B_R} e^{-\frac{\delta\lvert y-x)\rvert^2}{(1+\delta)\e^2}}\,\dd(\lambda\otimes \mu)-\frac{\lambda(B_R)\mu(B_R)}{|B_R|^2} \dd(x\otimes y)\\
\leq & \, \frac{\lambda(B_R)\mu(B_R)}{|B_R|^2} \int_{B_R \times \R^d} e^{-\frac{\delta |y|^2}{(1+\delta)\e^2}} \,  \dd(x\otimes y) + R^\alpha [\lambda]_{\alpha,R}\int_{\R^d\times B_R}e^{-\frac{\delta\lvert y-x)\rvert^2}{(1+\delta)\e^2}}\,\dd(x\otimes \mu(y)) \\
& + \, R^\alpha [\mu]_{\alpha,R}\frac{\lambda(B_R)}{|B_R|} \int_{\R^d\times B_R}e^{-\frac{\delta\lvert y-x)\rvert^2}{(1+\delta)\e^2}}\,\dd(x\otimes y) \, .
\end{align}
Rescaling and simplifying the integrals, we obtain the bound
\begin{align}
&\int_{B_R \times B_R} e^{-\frac{\delta\lvert y-x)\rvert^2}{(1+\delta)\e^2}}\,\dd(\lambda\otimes \mu)\\
\leq & \,  \left(\frac{\e^2(1+\delta)}{\delta}\right)^{\frac d2}\left(\frac{\lambda(B_R)\mu(B_R)}{|B_R|}+\lambda(B_R)R^\alpha [\mu]_{\alpha,R} + \mu(B_R)R^\alpha [\lambda]_{\alpha,R} \right)\int_{\R^d} e^{-\lvert x\rvert^2}\,\dd x\\
\leq& \left(\frac{\pi\e^2(1+\delta)}{\delta}\right)^{\frac d2}\left(2\pi^2\lambda(0)\mu(0)+2\pi\lambda(0)+2\pi\mu(0)\right)R^d.
\end{align}
This leaves us with the lower bound
\begin{align}
&\e^2\log\left(\frac{1}{\int_{B_R \times B_R} e^{-\frac{\delta\lvert y-x\rvert^2}{(1+\delta)\e^2}}\dd(\lambda\otimes \mu)}\right)\pi_\e(P_R) \\ \geq & \, -\e^2\log(M)\pi_\e(P_R) +\frac d2\e^2 \log(\e^{-2})\pi_\e(P_R)\, ,
\label{eq:entropiclowerbound2}
\end{align}
where the constant $M=M(\delta,R,\lambda(0),\mu(0))$ is given by
\begin{equation}
M= 2\left(\frac{\pi\e^2(1+\delta)}{\delta}\right)^{\frac d2}\left(\pi^2\lambda(0)\mu(0)+\pi\lambda(0)+\pi\mu(0)\right)R^d.
\end{equation}
Note that $M$ is increasing as a function of $R$.
Combining \eqref{eq:entropiclowerbound2}, \eqref{eq:lower_bound_entropy}, and \eqref{eq:entropiclowerbound1}, we obtain
\begin{align}
&\int_{\#_R} \log\left(\frac{\dd \pi_\e}{\dd(\lambda\otimes\mu)}\right)\frac{\dd \pi_\e}{\dd(\lambda\otimes\mu)}\dd(\lambda\otimes\mu)\\
\geq&\,  -\e^2\log(M)\pi_\e(P_R) +\frac d2\e^2 \log(\e^{-2})\pi_\e(P_R)  +\e^2 \pi_\e(\#_R)\log(\pi_\e(P_R)) \\&\, -\frac{\delta}{1+\delta}\int_{B_R \times B_R}|y-x|^2 \, \dd \pi_\e\, ,
\end{align}
which together with \eqref{eq:fullcostbound} gives us
\begin{align}
&\int_{\#_R} \lvert x-y\rvert^2 \dd \pi_\e  \\
\leq & \, \pi_\e(P_R) OT(\bar \lambda,\bar \mu)+ \pi_\e(\#_{\Lambda R})C_1\e^2 +\delta \int_{\#_{2R}}|x-y|^2\,\mathrm{d}\pi_\e \, ,
\end{align}
for some constant $C_1<\infty$ which depends on $R,\delta,\lambda(0)\,\mu(0)$, where we have used the fact that $P_R \subset \#_R \subset \#_{\Lambda R}$. Moreover, we note that $C_1$ is increasing as a function of $R$, since $M$ is. This completes the proof of the proposition.
\end{proof}

Thus, using Propositions~\ref{prop:verylong} and~\ref{prop:quasiMin1}, we have shown that Assumptions (i) and (ii) are satisfied for entropic optimal transport for any $\delta\in (0,1)$ as long as $\e^2/R^2$ is chosen to be sufficiently small which is always possible. Thus,~\eqref{eq:smallness} will be satisfied as long as ~\eqref{eq:smallnessentropic} is (for possibly different choices of $\e_1>0$).
\subsection*{Acknowledgements}
The authors would like to thank Felix Otto, for suggesting this problem, for several useful discussions during the course of this work, and for providing the main ideas for the proof of Lemma~\ref{lem:soft}. The authors are also grateful to Francesco Mattesini for many fruitful discussions during the course of this project. A large part of this research was carried out while the authors were researchers at the Max Planck Institute for Mathematics in the Sciences (MPI-MiS), Leipzig. Both authors are grateful to the MPI-MiS for its support and hospitality.

\bibliographystyle{amsalpha}
\bibliography{ref}

\end{document}